\documentclass{article}

\author{Mahmood Alizadeh}
\date{}

\title{Some notes on the $k$-normal elements and $k$-normal polynomials over finite fields}
\usepackage{amsfonts}
\usepackage{amsmath}
\usepackage{amssymb}

\newtheorem{theorem}{Theorem}[section]

\newtheorem{corollary}[theorem]{Corollary}
\newtheorem{definition}[theorem]{Definition}

\newtheorem{proposition}[theorem]{Proposition}

\newtheorem{notation}[theorem]{Notation}
\newtheorem{lemma}[theorem]{Lemma}
\newenvironment{proof}[1][Proof]{\noindent{\bf #1.} }{~\hfill \rule{1,5mm}{1,5mm}}

\begin{document}
\maketitle{
\begin{center}
 \small\it Department of Mathematics, College of Science,
Ahvaz Branch, Islamic Azad University, Ahvaz, Iran
\end{center}
\begin{center}
E-mail: alizadeh@iauahvaz.ac.ir
\end{center} \vspace{.5cm}
}

\begin{abstract}
Recently, the $k$-normal element over finite fields is defined and
characterized by Huczynska et al.. In this paper, the
characterization of $k$-normal elements, by using to give a
generalization of Schwartz's theorem, which allows us to check if an
element is a normal element,  is obtained. In what follows, in respect of
the problem of existence of a primitive 1-normal element in
$\mathbb{F}_{q^n}$ over $\mathbb{F}_{q}$, for all $q$ and $n$, had
been stated by Huczynska et al., it is shown that, in general, this
problem is not satisfied. Finally, a recursive method for
constructing $1$-normal polynomials of higher degree from a given
$1$-normal polynomial over $\mathbb{F}_{2^m}$ is given.
\end{abstract}
%\begin{keyword}
%\texttt{elsarticle.cls}\sep \LaTeX\sep Elsevier \sep template
%MSC[2010] 11T30, 11T06, 12E20
%\end{keyword}

\section{Introduction}\label{sec11}

Let $\mathbb{F}_{q}$, be the Galois field of order $q=p^{m}$, where
$p$ is a prime and $m$ is a natural number, and
$\mathbb{F}^{\ast}_{q}$ be its multiplicative group. A $normal$
basis of $\mathbb{F}_{q^n}$ over $\mathbb{F}_{q}$ is a basis of the
form $ N={\{\alpha, \alpha^{q}, ... , \alpha^{q^{n-1}}\}}$, i.e., a
basis that consists of the algebraic conjugates
 of a fixed element $\alpha\in{\mathbb{F}^{\ast}_{q^n}}$.
Recall that an element $\alpha\in\mathbb{F}_{q^n}$ is said to
generate a normal
 basis over $\mathbb{F}_{q}$ if its conjugates form a basis of $\mathbb{F}_{q^n}$
  as a vector space over $\mathbb{F}_{q}$. For convenience an element $\alpha$
  generating a normal basis, is called a $normal$ element.

A monic irreducible polynomial $F(x)\in\mathbb{F}_{q}[x]$ is called
{\it{normal polynomial}} or {\it{$N$-polynomial}} if its roots form
a normal basis or, equivalently, if they are linearly independent
over $\mathbb{F}_{q}$. The minimal polynomial of an element in a
normal basis ${\{\alpha, \alpha^{q}, ... , \alpha^{q^{n-1}}\}}$ is
$m(x)=\prod_{i=0}^{n-1}{(x-\alpha^{q^i})}\in \mathbb{F}_{q}[x]$,
 which is irreducible over  $\mathbb{F}_{q}$.

Normal bases have many applications, including  coding theory,
cryptography and computer algebra systems. The elements in a normal
basis are exactly the roots of some $N$-polynomial. Hence an
$N$-polynomial is just another way of describing
   a normal basis. The construction of normal
polynomials over finite fields is a challenging mathematical problem
\cite{8}.

 Recently, the $k$-normal element over finite fields is defined and characterized by
Huczynska, Mullen, Panario and Thomson \cite{6}. The following
definition for checking whether an element is a $k$-normal element,
is given by them.
\begin{definition}\label{def11}
For each $0\leq k\leq n-1$, the element $\alpha \in
\mathbb{F}_{q^n}$ is called a $k$-normal element if $deg(gcd(x^n-1,
\sum_{i=0}^{n-1}{\alpha^{q^i}}x^{n-1-i}))=k$.
\end{definition}
Using this mention, a normal element of $\mathbb{F}_{q^n}$ over
$\mathbb{F}_{q}$ is a 0-normal element. The existence of primitive
normal elements has been obtained in \cite{1} and \cite{2} for
sufficiently large $q$ and $n$. The {\it{Primitive Normal Basis
Theorem}} has been established for all $q$ and $n$ in \cite{10} and
\cite{4}.
Moreover, the existence of primitive $1$-normal elements is
established in \cite{6}, for sufficiently large $q^n$. Also, the
existence of primitive $1$-normal elements of $\mathbb{F}_{q^n}$
over $\mathbb{F}_{q}$, for all $q$ and $n$ is stated in (\cite{6},
Problem 6.2), as a problem.

In this paper, in Section \ref{sec11z}, some definitions, notes and
results which are useful for our study have been stated. Section
\ref{sec41} is devoted to characterization and construction of
$k$-normal elements (especially, 1-normal elements). Also, in this
section, in respect of the problem of existence of primitive
$1$-normal elements, for all $q$ and $n$, some value of $q$ and $n$
is given, such that the extension finite field $\mathbb{F}_{q^n}$ is
not consists of any primitive $1$-normal element. Finally, in
Section \ref{sec3}, by definition of $k$-normal polynomials, given
in Section \ref{sec11z}, a recursive method for construction
$1$-normal polynomials of higher degree from a given $1$-normal
polynomial over $\mathbb{F}_{2^m}$ is given.

\section{Preliminary notes}\label{sec11z}
We use the definitions, notations and results given by Huczynska
\cite{6}, Cohen\cite{3}, Gao \cite{5} and Kyuregyan \cite{8}, where
similar problems are considered.
 We need the following results for our further study.

Recall that $Tr_{\mathbb{F}_{q^n}|{\mathbb{F}_{q}}}(\alpha)$, the
trace of $\alpha$ in $\mathbb{F}_{q^n}$ over $\mathbb{F}_{q}$, is
given by
$Tr_{\mathbb{F}_{q^n}|{\mathbb{F}_{q}}}(\alpha)=\sum_{i=0}^{n-1}{\alpha^{q^i}}$.
For convince, $Tr_{\mathbb{F}_{q^n}|{\mathbb{F}_{q}}}$ is denoted by
$Tr_{q^n|{q}}$. To state further results, recall that the Frobenius
map
$$\sigma:\eta\rightarrow{\eta}^q,\hspace{.2cm}\eta\in \mathbb{F}_{q^n}$$
 is an automorphism of $\mathbb{F}_{q^n}$ that fixes $\mathbb{F}_{q}$.
In particular, $\sigma$ is a linear transformation of
$\mathbb{F}_{q^n}$ viewed as a vector space of
 dimension $n$ over $\mathbb{F}_{q}$. For any polynomial $g(x)={\sum_{i=0}^{n-1}g_{i}x^i}\in F_{q}[x]$,
define
\begin{align}
g(\sigma)\eta &=({\sum_{i=0}^{n-1}g_{i}\sigma^i})\eta
={\sum_{i=0}^{n-1}g_{i}\eta^{q^i}},\nonumber
\end{align}
 which is also a linear transformation on $F_{q^n}$.

The $null\hspace{.09cm} space$ of $g(\sigma)$ is defined to be the
set of all elements $\alpha\in \mathbb{F}_{q^n}$ such that
$g(\sigma)\alpha=0$,
 we also call it the null space of $g(x)$. For any element $\alpha\in \mathbb{F}_{q^n}$,
 the monic polynomial $g(x)\in \mathbb{F}_{q}[x]$ of the smallest degree such that $g(\sigma)\alpha=0$ is called the $\sigma$-order
 of $\alpha$. This polynomial is denoted by $Ord_{\alpha,\sigma}(x)$. Note that $Ord_{\alpha,\sigma}(x)$ divides any
  polynomial $h(x)$ annihilating $\alpha$. In particular,
  for every $\alpha\in F_{q^n}$, $Ord_{\alpha,\sigma}(x)$ divides $x^n-1$, the minimal or characteristic polynomial for $\sigma$.
\begin{proposition}\cite{6}\label{proc11}
Let $\alpha \in \mathbb{F}_{q^n}$. Then $\alpha$ is a $k$-normal
element over $\mathbb{F}_{q}$ if and only if
$deg(Ord_{\alpha,\sigma}(x)) = n-k.$
 \end{proposition}
\begin{proposition}\cite{3}\label{proc11z}
A necessary condition for existence a primitive element with zero
trace in $\mathbb{F}_{q^n}$ over $\mathbb{F}_{q}$ is $n\geq 3$,
except when $n=2$ and $q=3$. Also if $n\geq3$, such a primitive
element exists, except when $n=3$ and $q=4$.
 \end{proposition}
Recall that, an element $\alpha \in \mathbb{F}_{q^n}$ is called a
proper element of $\mathbb{F}_{q^n}$ over $\mathbb{F}_{q}$ if
$\alpha\notin \mathbb{F}_{q^v}$ for any proper divisor $v$ of $n$.
So by this mention, an element $\alpha \in \mathbb{F}_{q^n}$ is
called a proper $k$-normal element of $\mathbb{F}_{q^n}$ over
$\mathbb{F}_{q}$ if $\alpha$ is a $k$-normal element and a proper
element of $\mathbb{F}_{q^n}$ over $\mathbb{F}_{q}$.

\begin{definition}
A monic irreducible polynomial $P(x)\in \mathbb{F}_{q}$ of degree
$n$ is called a $k$-normal polynomial (or $N_k$-polynomial)over
$\mathbb{F}_{q}$ if its roots be the $k$-normal elements of
$\mathbb{F}_{q^n}$ over $\mathbb{F}_{q}$.
\end{definition}
By using the above definition, an $N$-polynomial is an
$N_{0}$-polynomial. Since, the proper $k$-normal elements of
$\mathbb{F}_{q^n}$ over $\mathbb{F}_{q}$ are the roots of a
$k$-normal polynomial of degree $n$ over $\mathbb{F}_{q}$, hence the
$k$-normal polynomials of degree $n$ over $\mathbb{F}_{q}$ is just
another way of describing the proper $k$-normal elements of
$\mathbb{F}_{q^n}$ over $\mathbb{F}_{q}$.

The following proposition is useful for constructing
$N_1$-polynomials over $\mathbb{F}_{2^m}$.

\begin{proposition}\cite{8}\label{the21xz}
Let $\delta\in\mathbb{F}^{\ast}_{2^s}$ and
$F_{1}(x)=\sum_{i=0}^{n}c_{i}x^{i}$ be an irreducible polynomial of
degree $n$ over $\mathbb{F}_{2^m}$. Also let $
Tr_{2^{m}|{2}}(\frac{c_{1}\delta}{c_{0}})=1$ and
$Tr_{2^{m}|{2}}(\frac{c_{n-1}}{\delta})=1$. Then the sequence
${(F_{u}(x))}_{u\geq1}$, defined by
\begin{align}\label{eq8tw}
 F_{u+1}(x)={x^{n2^{u-1}}}F_{u}(x+\delta^2x^{-1}) \hspace{1.5cm} u\geq
 1.\nonumber
\end{align}
 is a sequence of irreducible polynomials of degree $2^un$ over $\mathbb{F}_{2^m}$
for every $u\geq1$.
\end{proposition}

  \section{Characterization and construction of
$k$-normal elements and existence of primitive 1-normal
elements}\label{sec41} In this section, the characterization and
construction of $k$-normal elements (specially, of 1-normal
elements), and also a characterization of some extension finite
fields, consist a primitive 1-normal element is given. In the
continue, in respect of the problem of existence of the primitive
$1$-normal elements for all $q$ and $n$ (\cite{6},  problem 6.2),
some values of $q$ and $n$ is introduced, such that the extension
finite field $\mathbb{F}_{q^n}$ is not consists of any primitive
$1$-normal element over $\mathbb{F}_{q}$.

 Let $p$ denote the characteristic of $\mathbb{F}_{q}$ and let $n=n_{1}p^{e
}=n_{1}t$,  with $gcd(p,n_{1})=1$ and suppose that
 $x^{n}-1$ has the following factorization in $\mathbb{F}_{q}[x]:$
\begin{equation}
x^{n}-1=(\varphi_{1}(x)\varphi_{2}(x)\cdots\varphi_{r}(x))^{t},\label{eq310}
\end{equation}
where $\varphi_{i}(x)\in \mathbb{F}_{q}[x]$ are the distinct
irreducible factors of $x^{n}-1$. For each $s$, $0\leq s<n$, let
there is a $u_s>0$ such that, $R_{s,1}(x)$, $R_{s,2}(x)$, $\cdots$,
$R_{s,{u_s}}(x)$ be all of the $s$ degree polynomials dividing
$x^{n}-1$. So, from (\ref{eq310}) we can write
$R_{s,i}(x)=\prod_{j=1}^{r}{{\varphi_{j}}^{t_{ij}}(x)}$, for each
$1\leq i\leq u_{s}$, $0\leq t_{ij}\leq t$.
  Let
\begin{equation}
\phi_{s,i}(x)=\frac{x^{n}-1}{R_{s,i}(x)}, \label{eq311}
\end{equation}
for $1\leq i \leq u_{s} $. Then, there is a useful characterization
 of the $k$-normal elements in $\mathbb{F}_{q^n}$ over $\mathbb{F}_{q}$ as follows.

\begin{lemma}\label{lem1}
An element $\alpha \in \mathbb{F}_{q^n}$ is a $k$-normal element if
and only if, there is a $j$, $1\leq j\leq u_k$, such that
\begin{equation}\label{eq1}
\phi_{k,{j}}(\sigma)\alpha=0,
\end{equation}
 and also
\begin{equation}\label{eq2}
 \phi_{s,{i}}(\sigma)\alpha\neq0,
\end{equation}
 for each $s$, $k<s<n$, and $1\leq i \leq u_{s} $.
\end{lemma}
\begin{proof}
According to Proposition \ref{proc11} , $\alpha$ is a $k$-normal
element of $\mathbb{F}_{q^n}$ over $\mathbb{F}_{q}$ if and only if
$deg(Ord(\alpha))=n-k$. This is true if and only if the Equations
(\ref{eq1}) and (\ref{eq2}) hold.
 \end{proof}

In the case $k=0$, the Lemma \ref{lem1} is obtained by Schwartz
\cite{14}(see also \cite{11}, Theorem 4.10). Actually, Lemma
\ref{lem1} is a generalization of Schwartz's theorem. The definition
\ref{def11} does give us a method to test if an element is a
$k$-normal element. In fact, the Lemma \ref{lem1} dose give us
another way to check if an element is $k$-normal.

The following results are some immediate consequences of Lemma
\ref{lem1} .

\begin{theorem}\label{thm321k}
Let $n=n_{1}p^e$, for some $e\geq1$ and $a,b\in
\mathbb{F}^{\ast}_{q}$. Then the element $\alpha$ is a proper
$1$-normal element of $\mathbb{F}_{q^n}$ over $\mathbb{F}_{q}$ if
and only if $a+b\alpha$ is a proper $1$-normal element of
$\mathbb{F}_{q^n}$ over $\mathbb{F}_{q}$.
\end{theorem}
\begin{proof}
 Let $n=n_{1}p^{e}=n_{1}t$ , then by (\ref{eq310}), $x^{n}-1$ has the
following factorization in $\mathbb{F}_{q}[x]:$
$$
x^{n}-1=(x^{n_{1}}-1)^{t}=(\varphi_{1}(x)\varphi_{2}(x)\cdots\varphi_{r}(x))^{t}.
$$
Let
\begin{equation}\label{eq2jhs}
 R_{1,1}(x)=\varphi_{1}(x)=x-1.
\end{equation}
 Set for each $s$, $0<s<n$ and
$1\leq i\leq u_{s}$, except when $s=i=1$,
\begin{align}
\phi_{s,i}(x)&=\frac{x^{n}-1}{R_{s,i}(x)}\nonumber\\
&=(x-1)^{t}M_{s,i}(x)\nonumber\\
&=(x-1)M'_{s,i}(x),\nonumber
\end{align}
 where
$$M'_{s,i}(x)=(x-1)^{t-1}M_{s,i}(x)=\sum_{v=0}^{n-s-1}t'_{iv}x^{v}.$$
Hence
\begin{equation}\label{eq2jh}
\phi_{s,i}(x)=\sum_{v=0}^{n-s-1}t'_{iv}x^{v+1}-\sum_{v=0}^{n-s-1}t'_{iv}x^{v}.
\end{equation}
From (\ref{eq2jh}), we obtain
\begin{align}
\phi_{s,{i}}(\sigma)(a+b\alpha)&=\sum_{v=0}^{n-s-1}t'_{iv}(a+b\alpha)^{q^{v+1}}-\sum_{v=0}^{n-s-1}t'_{iv}(a+b\alpha)^{q^{v}},\nonumber\\
&=a\sum_{v=0}^{n-s-1}t'_{iv}+b\sum_{v=0}^{n-s-1}t'_{iv}
\alpha^{q^{v+1}}-a\sum_{v=0}^{n-s-1}t'_{iv}-b\sum_{v=0}^{n-s-1}t'_{iv}\alpha^{q^{v}}\nonumber\\
&=b(\sum_{v=0}^{n-s-1}t'_{iv}\alpha^{q^{v+1}}-
\sum_{v=0}^{n-s-1}t'_{iv}\alpha^{q^{v}})
=b~\phi_{s,{i}}(\sigma)(\alpha), \label{eq337bg}
\end{align}
for each $s$, $0<s<n$ and $1\leq i\leq u_{s}$, except when $s=i=1$.
On the other hand, in the case $s=i=1$, from (\ref{eq2jhs}), we have
$$
\phi_{1,{1}}(x)=\frac{x^{n}-1}{R_{1,1}(x)}=\frac{x^{n}-1}{x-1}=x^{n-1}+x^{n-2}+\cdots+x+1=\sum_{i=0}^{n-1}x^{i}.$$So
\begin{align}
\phi_{1,{1}}(\sigma)(a+b\alpha)&=\sum_{i=0}^{n-1}(a+b\alpha)^{q^{i}}\nonumber\\
&=\sum_{i=0}^{n-1}a+b\sum_{i=0}^{n-1}\alpha^{q^{i}}\nonumber\\
&=na+b~\phi_{1,{1}}(\sigma)(\alpha)\nonumber\\
 &=b~\phi_{1,{1}}(\sigma)(\alpha).
 \label{eq337t}
\end{align}

From (\ref{eq337bg}) and (\ref{eq337t}), it follows that
$$\phi_{s,{i}}(\sigma)(a+b\alpha)=b~\phi_{s,{i}}(\sigma)(\alpha)$$
for each $s$, $0<s<n$ and $1\leq i\leq u_{s}$. By Lemma \ref{lem1},
the proof is completed.
\end{proof}

\begin{theorem}\label{thm322}
Suppose that $\alpha$ is a proper element of $\mathbb{F}_{q^2}$.
Then $\alpha$ is $1$-normal over $\mathbb{F}_{q}$ if and only if
$Tr_{q^2|{q}}(\alpha)=0$.
\end{theorem}

\begin{proof}
By (\ref{eq310}), $x^2-1$ has the following factorization in
$\mathbb{F}_{q}[x]$:
$$x^2-1=(x-1)(x+1).$$
So, we have $R_{1,{1}}(x)=x-1$ and $R_{1,{2}}(x)=x+1$. Hence,
$\phi_{1,{1}}(x)=x+1$ and $\phi_{1,{2}}(x)=x-1$. Note that, by Lemma
\ref{lem1}, to complete the proof, it is enough that
\begin{equation}\label{eqw2}
\phi_{1,{1}}(\sigma)\alpha={\alpha^{q}}+\alpha=Tr_{q^2|{q}}(\alpha)=0.
\end{equation}
This is satisfied by hypothesis.
\end{proof}
\begin{theorem}\label{thm323}
Let $n$ be a prime different from $p$ and $q$ a primitive element
modulo $n$. Suppose that $\alpha$ is a proper element of
$\mathbb{F}_{q^n}$. Then $\alpha$ is a $1$-normal element of
$\mathbb{F}_{q^n}$ over $\mathbb{F}_{q}$ if and only if
$Tr_{q^n|{q}}(\alpha)=0$.
\end{theorem}
\begin{proof}
Note that $$x^n-1=(x-1)(x^{n-1}+\cdots+x+1).$$ Since $q$ is a
primitive modulo $n$, $x^{n-1}+\cdots+x+1$ is an irreducible
polynomial over $\mathbb{F}_{q}$. Hence, from (\ref{eq310}), we have
$R_{1,{1}}(x)=x-1$ and $R_{n-1,{1}}(x)=x^{n-1}+\cdots+x+1$. So,
$\phi_{1,{1}}(x)=x^{n-1}+ \cdots +x+1$ and $\phi_{n-1,{1}}(x)=x-1$.
Thus , by Lemma \ref{lem1}, $\alpha\in\mathbb{F}_{q^n}$ is a
$1$-normal element of $\mathbb{F}_{q^n}$ over $\mathbb{F}_{q}$ if
and only if
\begin{equation}\label{eqw2}
\phi_{1,{1}}(\sigma)\alpha=\sum_{i=0}^{n-1}{\sigma^i{(\alpha)}}=\sum_{i=0}^{n-1}{\alpha^{q^i}}=Tr_{q^n|{q}}(\alpha)=0.
\end{equation}
and
\begin{equation}\label{eqw3}
\phi_{n-1,{1}}(\sigma)\alpha={\alpha^{q}}-\alpha\neq0.
\end{equation}
The expressions (\ref{eqw2}) and (\ref{eqw3}) are satisfied by
hypothesis. This completes the proof.
\end{proof}

In the case, the $0$-normal elements, similar problems for Theorems
\ref{thm321k} and \ref{thm323} had been established by Kyuregyan
\cite{8} and
 Pei, Wang and Omura \cite{12}, respectively.

 The following corollary which characterize some of
extension finite fields, consist a primitive $1$-normal element, is
an immediate consequence of Theorems \ref{thm322}, \ref{thm323} and
Proposition \ref{proc11z}.

\begin{corollary}\label{core2z}
Let $n$ be a prime different from $p$ and $q$ a primitive element
modulo $n$. Then, A necessary condition for existence a primitive
$1$-normal element in $\mathbb{F}_{q^n}$ over $\mathbb{F}_{q}$ is
$n\geq 3$, except when $n=2$ and $q=3$. Also if $n\geq3$, such a
primitive $1$-normal element exists, except when $n=3$ and $q=4$.
\end{corollary}
In respect of the problem of existence of a primitive $1$-normal
element in $\mathbb{F}_{q^n}$ over $\mathbb{F}_{q}$ , for all $q$
and $n$, the following notation, which is an immediate consequence
of Corollary \ref{core2z} is given.
\begin{notation}
By Corollary \ref{core2z}, it is shown that, there are some
extension finite fields $\mathbb{F}_{q^n}$, with out any primitive
1-normal element over $\mathbb{F}_{q}$. For this, for example, we
can refer to the finite fields $\mathbb{F}_{4}$ over
$\mathbb{F}_{2}$, $\mathbb{F}_{4^3}$ over $\mathbb{F}_{4}$ and
$\mathbb{F}_{q^2}$ over $\mathbb{F}_{q}$, for each odd prime power
$q\neq3$. Consequently, it is shown that, the problem of existence
of a primitive $1$-normal element in $\mathbb{F}_{q^n}$ over
$\mathbb{F}_{q}$, for all $q$ and $n$ is not satisfied.
\end{notation}

\section{Recursive construction $N_{1}$-polynomials}\label{sec3}
In this section, we establish a theorem which will show how Lemma
\ref{lem1} and Proposition \ref{the21xz} can be applied to produce
infinite sequences of $N_{1}$-polynomials over $\mathbb{F}_{2^m}$.
Nothing that
$\phi_{s,{i}}(\sigma)\alpha=\sum_{i=0}^{n-s}{a_i\alpha^{q^i}}$ for
any polynomial $\phi_{s,{i}}(x)=\sum_{i=0}^{n-s}{a_ix^{i}}$, we can
restate Lemma \ref{lem1} as follows.
\begin{proposition}\label{pro3122}
Let $F(x)$ be an irreducible polynomial of degree $n$ over
$\mathbb{F}_{q}$ and $\alpha$ be a root of it. Let $x^{n}-1$ factors
as (\ref{eq310}) and let $\phi_{s,i}(x)$ be as in (\ref{eq311}).
Then $F(x)$ is a $N_{k}$-polynomial over $\mathbb{F}_{q}$ if and
only if, there is $j$, $1\leq j\leq u_k$, such that

$$L_{\phi_{k,{j}}}(\alpha)=0, $$
and also
$$L_{\phi_{s,{i}}}(\alpha)\neq 0,$$
for each $s$, $k<s<n$, and $1\leq i \leq u_{s} $, where $u_{s}$ is
the number of all $s$ degree polynomials dividing $x^{n}-1$ and
 $L_{\phi_{s,{i}}}(x)$ is the linearized polynomial defined by
$$L_{\phi_{s,{i}}}(x)=\sum_{v=0}^{n-s}t_{iv}x^{q^{v}} ~ if ~
\phi_{s,{i}}(x)=\sum_{v=0}^{n-s}t_{iv}x^{v}.$$
\end{proposition}

The following theorem dose give us a recursive method for
construction $1$-normal polynomials of higher degree from a given
$1$-normal polynomial over $\mathbb{F}_{2^m}$. A similar result, in
the case $0$-normal polynomials has been established by Kyuregyan in
\cite{8}. We use of an analogues technique, which used by Kyuregyan,
in the proof of theorem
\begin{theorem}\label{the21}
Let $\delta\in\mathbb{F}^{\ast}_{2^s}$ and
$F_{1}(x)=\sum_{i=0}^{n}c_{i}x^{i}$ be $N_{1}$-polynomial of degree
the even integer  $n$ over $\mathbb{F}_{2^m}$. Also let $
Tr_{2^{m}|{2}}(\frac{c_{1}\delta}{c_{0}})=1$ and
$Tr_{2^{m}|{2}}(\frac{c_{n-1}}{\delta})=1$. Then the sequence
${(F_{u}(x))}_{u\geq1}$, defined by
\begin{equation}\label{eq8t}
 F_{u+1}(x)={x^{n2^{u-1}}}F_{u}(x+\delta^2x^{-1}) \hspace{1.5cm} u\geq 1.
\end{equation}
 is a sequence of $N_{1}$-polynomials of degree $2^un$ over $\mathbb{F}_{2^m}$
for every $u\geq1$.
\end{theorem}
\begin{proof}
 Since $F_{1}(x)$ is an irreducible polynomial over
$\mathbb{F}_{2^m}$, so Proposition \ref{the21xz} and the hypothesis
imply that
$F_{2}(x)$ is irreducible over $\mathbb{F}_{2^m}$.\\
Let $\alpha\in\mathbb{F}_{2^{mn}}$ be a root of $F_{1}(x)$. Since
$F_{1}(x)$ is an $N_{1}$-polynomial of degree $n$ over
$\mathbb{F}_{2^m}$ by the hypothesis, so
$\alpha\in\mathbb{F}_{2^{mn}}$ is a $1$-normal element over
$\mathbb{F}_{2^m}$.

 Let $n=n_{1}2^{e}$, where $n_{1}$ is
a non-negative integer with $gcd(n_{1},2)=1$ and $e\geq0$. For
convenience we denote $2^e$ by $t$. Let $x^n+1$ has the following
factorization in $\mathbb{F}_{2^m}[x]$:
\begin{equation}
x^{n}+1={(x^{n_{1}}+1)}^t=(\varphi_{1}(x)\varphi_{2}(x)\cdots\varphi_{r}(x))^{t},
\label{eq333}
\end{equation}
 where
the polynomials $\varphi_{i}(x)\in\mathbb{F}_{q}[x]$ are the
distinct irreducible factors of $x^{n_{1}}+1$. From (\ref{eq333}),
for each $1<s<n$, that, there is a $u_s>0$ such that, $R_{s,1}(x)$,
$R_{s,2}(x)$, $\cdots$, $R_{s,{u_s}}(x)$ are all of the $s$ degree
polynomials dividing $x^{n}+1$, we can write
$R_{s,i}(x)=\prod_{j=1}^{r}{{\varphi_{j}}^{t_{ij}}(x)}$, for each
$1\leq i\leq u_{s}$, $0\leq t_{ij}\leq t$.
  Let
\begin{equation}\label{eq311}
\phi_{s,i}(x)=\frac{x^{n}+1}{R_{s,i}(x)}=\sum_{v=0}^{n-s}{t_{iv}{x^{v}}},
\end{equation}
for $1\leq i \leq u_{s} $. Since $F_{1}(x)$ is an $N_{1}$-polynomial
of degree $n$ over $\mathbb{F}_{2^m}$, so by Proposition
\ref{pro3122}, there is a $j$, $1\leq j\leq u_1$, such that
\begin{equation}\label{eq311m}
L_{\phi_{1,{j}}}(\alpha)=0,
\end{equation}
 and also
 \begin{equation}\label{eq311n}
L_{\phi_{s,{i}}}(\alpha)\neq 0,
\end{equation}
for each $1<s<n$, which there is a $u_{s}>0$ such that,
$R_{s,1}(x)$, $R_{s,2}(x)$, $\cdots$, $R_{s,{{u_{s}}}}(x)$ be all of
the $s$ degree polynomials dividing $x^{n}+1$, and $1\leq i \leq
u_{s} $.

Now, we proceed by proving that $F_{u+1}(x)$ is a $1$-normal
polynomial. Obviously, by Proposition \ref{the21xz} and the
hypothesis, $F_{u+1}(x)$ is irreducible over $\mathbb{F}_{2^m}$. Let
$\alpha_{u+1}$ be a root of $F_{u+1}(x)$.
Note that by (\ref{eq333}), the polynomial $x^{2^un}+1$ has the
following factorization in $\mathbb{F}_{2^m}[x]$:
 \begin{equation}\label{eq311nf}
x^{2^un}+1=(\varphi_{1}(x)\cdot \varphi_{2}(x)\cdot \cdots \cdot
\varphi_{r}(x))^{2^ut},
\end{equation}
 where $\varphi_{j}(x)\in
\mathbb{F}_{2^m}[x]$ are distinct irreducible factors of
$x^{n_{1}}+1$.

 For each $1<s'<2^un$, let there is a $u'_{s'}>0$
such that, $R'_{s',1}(x)$, $R'_{s',2}(x)$, $\cdots$,
$R'_{s,{{u'_{s'}}}}(x)$ be all of the $s'$ degree polynomials
dividing $x^{2^un}+1$. So, from (\ref{eq311nf}) we can write
$R'_{s',i'}(x)=\prod_{j=1}^{r}{{\varphi_{j}}^{t'_{i'j}}(x)}$, for
each $1\leq i'\leq u'_{s'}$, $0\leq t'_{i'j}\leq 2^ut$.
  Let
\begin{equation}
H'_{s',i'}(x)=\frac{x^{2^un}+1}{R'_{s',i'}(x)} \label{eq3111}
\end{equation}
for $1\leq i' \leq u'_{s'} $. By Proposition \ref{pro3122},
$F_{u+1}(x)$ is an $N_{1}$-polynomial of degree $2^un$ over
$\mathbb{F}_{2^m}$ if and only if, there is a $j$, $1\leq j\leq
u'_1$, such that

$$L_{H'_{1,{j}}}(\alpha_{u+1})=0, $$
and also
$$L_{H'_{s',{i'}}}(\alpha_{u+1})\neq 0,$$
for each $1<s'<2^un$, which there is a $u'_{s'}>0$ such that,
$R'_{s',1}(x)$, $R'_{s',2}(x)$, $\cdots$, $R'_{s,{{u'_{s'}}}}(x)$ be
all of the $s'$ degree polynomials dividing $x^{2^un}+1$, and $1\leq
i' \leq u'_{s'} $.

Let $H_{s,i}(x)=\frac{x^{2^un}+1}{R_{s,i}(x)}$, for each $1<s<n$,
and $1\leq i \leq u_{s}$, satisfied in the stated conditions.
 We note that by (\ref{eq311})

\begin{align}\label{eq31112}
H_{s,i}(x)&=\frac{x^{2^un}+1}{R_{s,i}(x)}\nonumber\\
&=\frac{(x^n+1)(\sum_{j=0}^{2^u-1}{x^{jn}})}{R_{s,i}(x)}\nonumber\\
&=(\sum_{j=0}^{2^u-1}{x^{jn}})\phi_{s,i}(x)\nonumber\\
&=\sum_{v=0}^{n-s}{t_{iv}{(\sum_{j=0}^{2^u-1}{x^{jn+v}}})}.
\end{align}
It follows that
\begin{equation}\label{eq31112}
L_{H_{s,i}}(\alpha_{u+1})=\sum_{v=0}^{n-s}{t_{iv}{(\sum_{j=0}^{2^u-1}{{{\alpha}^{2^{jmn}}_{u+1}}}})}^{2^{mv}}.
\end{equation}

From (\ref{eq8t}), if $\alpha_{y+1}$ is a root of $F_{y+1}(x)$, then
$\alpha_{y+1}+\delta^2{\alpha^{-1}_{y+1}}$ is a root of  $F_{y}(x)$,
and therefore, it may assume that
\begin{equation}\label{eq31112x}
\alpha_{y+1}+\delta^2{\alpha^{-1}_{y+1}}=\alpha_{y}
\end{equation}
for all $1\leq y\leq u$, where $\alpha_{1}=\alpha$.

Now, by (\ref{eq31112x}) and observing that $F_{y}(x)$ is an
irreducible polynomial of degree $2^{y-1}n$ over
$\mathbb{F}_{2^{m}}$, we obtain

\begin{equation}\label{eq31112y}
\alpha^{2^{2^{y-1}mn}}_{y+1}+\delta^2\alpha^{{-2}^{2^{y-1}mn}}_{y+1}=\alpha^{2^{2^{y-1}mn}}_{y}=\alpha_{y}.
\end{equation}
So from (\ref{eq31112x}) and (\ref{eq31112y}) we have

\begin{align}\label{eq31112y}
(\alpha^{2^{2^{y-1}mn}}_{y+1}+\alpha_{y+1})+\delta^2(\alpha^{{-2}^{2^{y-1}mn}}_{y+1}+{\alpha^{-1}_{y+1}})&=0.\nonumber
\end{align}
Or
\begin{equation}\label{eq31112z}
(\alpha^{2^{2^{y-1}mn}}_{y+1}+\alpha_{y+1})(1+\frac{\delta^2}{{\alpha}_{y+1}^{2^{2^{y-1}mn}+1}})=0.
\end{equation}
Also observing that $F_{y+1}(x)$ is an irreducible polynomial of
degree $2^yn$ over $\mathbb{F}_{2^{m}}$, we have
$\alpha^{2^{2^{y-1}mn}}_{y+1}+\alpha_{y+1}\neq0$, and so
${\alpha}_{y+1}^{2^{2^{y-1}mn}}=\delta^2\alpha^{-1}_{y+1}$. It
follows that
\begin{equation}\label{eq31112z1}
\alpha^{2^{2^{y-1}mn}}_{y+1}+\alpha_{y+1}=\delta^2\alpha^{-1}_{y+1}+\alpha_{y+1}=\alpha_{y}
\end{equation}
for each $y$, $1\leq y \leq u$. Hence by (\ref{eq31112z1}) we obtain

\begin{align}\label{eq31112z2}
\sum_{j=0}^{2^u-1}{{{\alpha}^{2^{jmn}}_{u+1}}}&=\sum_{j=0}^{{2^{u-1}}-1}{{(\alpha^{2^{2^{u-1}mn}}_{u+1}+\alpha_{u+1})}^{2^{jmn}}}=\sum_{j=0}^{{2^{u-1}}-1}{{\alpha^{2^{jmn}}_{u}}}\nonumber\\
&=\cdots=\sum_{j=0}^{{2^{u-z}}-1}{{{\alpha}^{2^{jmn}}_{u-z+1}}}=\sum_{j=0}^{{2^{u-z}}-1}{{(\alpha^{2^{2^{u-z-1}mn}}_{u-z+1}+\alpha_{u-z+1})}^{2^{jmn}}}\nonumber\\
&=\sum_{j=0}^{{2^{u-z}}-1}{{{\alpha}^{2^{jmn}}_{u-z}}}=\cdots=\sum_{j=0}^{1}{{{\alpha}^{2^{jmn}}_{2}}}=\alpha_{1}=\alpha.\nonumber
\end{align}
Therefore, $\sum_{j=0}^{2^u-1}{{{\alpha}^{2^{jmn}}_{u+1}}}=\alpha$.
By substituting $\alpha$ for
$\sum_{j=0}^{2^u-1}{{{\alpha}^{2^{jmn}}_{u+1}}}$ in (\ref{eq31112}),
we get
\begin{equation}\label{eq31112z3}
L_{H_{s,i}}(\alpha_{u+1})=\sum_{v=0}^{n-s}{t_{iv}{\alpha}^{2^{mv}}}=L_{\phi_{s,i}}(\alpha)
\end{equation}
 for each $s$, $1\leq s<n$, which there is a $u_s>0$ such that,
$R_{s,1}(x)$, $R_{s,2}(x)$, $\cdots$, $R_{s,{u_s}}(x)$ are all of
the $s$ degree polynomials dividing $x^{n}+1$.

So, by (\ref{eq311m}) and (\ref{eq31112z3}), there is a $j$,
$1<j<u'_{1}=u_{1}$, such that
$$L_{H'_{1,j}}(\alpha_{u+1})=L_{H_{1,j}}(\alpha_{u+1})=L_{\phi_{1,j}}(\alpha)=0$$

Now we clime that, for each $1<s'<2^un$, which there is a
$u'_{s'}>0$, such that, $R'_{s',1}(x)$, $R'_{s',2}(x)$, $\cdots$,
$R'_{s,{{u'_{s'}}}}(x)$ be all of the $s'$ degree polynomials
dividing $x^{2^un}+1$, and $1\leq i' \leq u'_{s'} $ we have
$$L_{H'_{s',{i'}}}(\alpha_{u+1})\neq 0.$$
Setting $t'_{i'j}=t$ for each $ t'_{i'j}>t$ in
$R'_{s',i'}(x)=\prod_{j=1 }^{r}{{\varphi_{j}}^{t'_{i'j}}(x)}$, we
get

 $$(R'_{s',i'}(x))_{\mid_{t'_{i'j}=t for t'_{i'j}>t}}=\prod_{j=1,t'_{i'j}=t for t'_{i'j}>t
}^{r}{{\varphi_{j}}^{t'_{i'j}}(x)}= \prod_{j=1
}^{r}{{\varphi_{j}}^{t_{i_{0}j}}(x)}=R_{s_{0},i_{0}}(x),$$ for some
$1<s_{0}<n$ and $1\leq i_{0}\leq u_{s_{0}}$. Hence,
$R_{s_{0},i_{0}}(x)$ divide $R'_{s',i'}(x)$ and so $H'_{s',i'}(x)$
divide $H_{s_{0},i_{0}}(x)$. Consequently,
$L_{H'_{s',i'}}(\alpha_{u+1})$ divide
$L_{H_{s_{0},i_{0}}}(\alpha_{u+1})$. Therefore, by (\ref{eq311n})
and (\ref{eq31112z3}), the clime is true. The proof is completed.
\end{proof}

\end{document}